\documentclass{amsart}
\usepackage{amsfonts}
\usepackage{graphics}

\newtheorem{theorem}{Theorem}
\theoremstyle{plain}

\newtheorem{corollary}{Corollary}

\newtheorem{definition}{Definition}

\newtheorem{lemma}{Lemma}

\newtheorem{remark}{Remark}

\numberwithin{equation}{section}
\input{tcilatex}

\begin{document}
\title[An approach for minimal surface family passing a curve]{An approach
for minimal surface family passing a curve}
\author{Sedat Kahyao\u{g}lu}
\address{Ondokuz May\i s University}
\email{sedatkahyaoglu@hotmail.com}
\author{Emin Kasap}
\curraddr{Ondokuz May\i s University}
\email{kasape@omu.edu.tr}
\date{August 15, 2014}

\begin{abstract}
We investigate minimal surfaces passing a given curve in 
{}$\mathbb{R}${}%
${{}^3}$%
. Using the Frenet frame of a given curve and isothermal parameter, we derive
the necessary and sufficient condition for minimal surface. Also we derive
the parametric representation of two minimal surface families passing a
circle and a helix as examples.
\end{abstract}

\maketitle

\section{Introduction}

\bigskip

A surface is called minimal if its mean curvature vanishes everywhere \cite%
{carmo}.Minimal surfaces are among the most important objects studied in
differential geometry. The theory of minimal surfaces in Euclidean spaces $%
\mathbb{R}
^{3}$ has its roots in the calculus of variations developed by Euler and
Lagrange in the 18-th century and in later investigations by Enneper,
Scherk, Schwarz, Riemann and Weierstrass in the 19-th century. During the
years, many great mathematicians have contributed to this theory.

Early, two problems were posed from the Belgian physicist Plateau and from
the Swedish mathematician Bj\"{o}rling in the 19-th century for minimal
surfaces passing a curve. Plateau set the problem as 'find a minimal surface 
$M$ having a curve $C$ as boundary'. This problem was a natural outgrowth of
his physical experiments on soap films. Plateau's problem was first solved by
Douglas and Rado. Bj\"{o}rling considered is to find the minimal surface
passing through a given non-closed analytic curve $C$, with given tangent
planes along $C$. The problem was inquired and solved by Bj\"{o}rling. A
solution of this problem always exists and is unique.This solution of Bj%
\"{o}rling's problem always applicable one to find the minimal surface whenever
one of its geodesic lines or one of its asymptotic lines or one of its lines
of curvature is known. One can obtain and detailed description of minimal
surfaces, Plateau's and Bj\"{o}rling's problems from \cite{osserman} and \cite%
{oprea}

Recently, parametric representation of a surface family passing a given
geodesic curve was studied in \cite{wang} which using the Ferenet frame of a
given curve. In like manner, Li et al. \cite{line of curvature}and Bayram et
al.\cite{asymptotic curve} studied the necessary and sufficient condition
for line of curvature and asymptotic curve, respectively. Li et al. \cite%
{approximation} studied the approximation minimal surface with geodesics by
using Dirichlet function.

In this paper, without limitation for curve such that it becomes geodesic
lines or asymptotic lines or lines of curvature on minimal surface, we
investigate the parametric representation of minimal surface family passing
a given curve. We utilize the Frenet frame of given curve and isothermal
parameter on minimal surface. We derive the necessary and sufficient
condition for a surface which satisfying both passing a given curve and
becomes minimal.

In section two, we give definitions of isothermal parameter of a surface and
Frenet frame of a curve. In section three, we derive the necessary and
sufficient condition for minimal surface. In section four and five, we
present parametric representations of minimal surface families passing a
given circle and a helix, respectively.

\section{Preliminaries}

We have following tools for minimal surfaces and curves.

\begin{definition}
\cite{carmo}Let $x:x(s,t)\subset 
\mathbb{R}
^{3}$ be a reguler surface where $s\in I\subset $ $%
\mathbb{R}
^{3}$and $t\in J\subset $ $%
\mathbb{R}
^{3}$. A parametrization $x:x(s,t)\subset 
\mathbb{R}
^{3}$ is called isothermal if 
\begin{equation}
<x_{s},x_{s}>=<x_{t},x_{t}>\text{and}<x_{s},x_{t}>=0.  \label{isothermal}
\end{equation}
\end{definition}

We utilize isothermal parameters for parameterization of minimal surfaces.
This requirement is not restriction because of the following lemma.

\begin{lemma}
\cite{oprea}Isothermal parameters exist on any minimal surface in $%
\mathbb{R}
^{3}$.
\end{lemma}

The mean curvature of a surface can be calculated by 
\begin{equation}
H=\frac{Eg-2Ff+Ge}{EG-F^{2}}  \label{mean}
\end{equation}

where

\begin{equation*}
E=<x_{s},x_{s}>,F=<x_{s},x_{t}>,G=<x_{t},x_{t}>
\end{equation*}%
\begin{equation*}
e=<n,x_{ss}>,f=<n,x_{st}>,g=<n,x_{tt}>\text{.}
\end{equation*}

and $n$ is unit normal of the surface defined by 
\begin{equation*}
n(s,t)=\frac{x_{s}(s,t)\times x_{t}(s,t)}{\left\Vert x_{s}(s,t)\times
x_{t}(s,t)\right\Vert }
\end{equation*}%
If we apply isothermal conditions on \ref{mean} then we have the following
corollary.

\begin{corollary}
\cite{carmo}Let $x$ is isothermal. Then $x$ is minimal if and only if its
coordinate functions are harmonic i.e. 
\begin{equation}
x_{ss}+x_{tt}=0.  \label{harmonic}
\end{equation}
\end{corollary}

Let $r:r(s)\subset 
\mathbb{R}
^{3}$ be an arclenght parametrized curve that is $\left\Vert r^{\prime
}(s)\right\Vert =1$ . We assume that $r^{\prime \prime }(s)\neq 0$ for some $%
s\in I$, otherwise the curve $r$ corresponds to a straight lines.The Frenet
frame of a regular curve $r$ is defined $\{T(s),N(s),B(s)\}$ where $%
T(s)=r^{\prime }(s)$, $N(s)=\frac{T^{\prime }(s)}{\left\Vert T^{\prime
}(s)\right\Vert }$, $B(s)=T(s)\times N(s)$, are called unit tangent, unit
normal and binormal of $r$, respectively. The derivatives of the Franet
frame are given by

\begin{align}
T^{\prime }(s)& =\kappa (s)N(s)  \notag \\
N^{\prime }(s)& =-\kappa (s)T(s)+\tau (s)B(s) \\
B^{\prime }(s)& =-\tau (s)N(s)
\end{align}%
where $\kappa (s)$ and $\tau (s)$ are the curvature and torsion of the curve 
$r$, respectively.

By utilizing the Frenet frame, Wang et al.\cite{wang}, construct a surface
pencil that posseses $r$ as a common geodesic . They gave parametric
representation of surface as

\begin{equation}
x(s,t)=r(s)+u(s,t)T(s)+v(s,t)N(s)+w(s,t)B(s)  \label{surpen}
\end{equation}%
where $u(s,t),v(s,t),w(s,t)$ are smooth functions with

\begin{equation}
u(s,t_{0})=v(s,t_{0})=w(s,t_{0})=0\text{.}  \label{isopara}
\end{equation}

The unit normal of the surface is 
\begin{equation*}
n(s,t)=\frac{(\phi _{1}(s,t)T(s)+\phi _{2}(s,t)N(s)+\phi _{3}(s,t)B(s))}{%
\sqrt{\left( \phi _{1}(s,t)\right) ^{2}+\left( \phi _{2}(s,t)\right)
^{2}+\left( \phi _{3}(s,t)\right) ^{2}}}
\end{equation*}%
where 
\begin{eqnarray*}
\phi _{1}(s,t) &=&w_{t}(v_{s}+\kappa u-\tau w)-v_{t}(w_{s}+\tau v), \\
\phi _{2}(s,t) &=&u_{t}(w_{s}+\tau v)-w_{t}(1+u_{s}-\kappa v), \\
\phi _{3}(s,t) &=&v_{t}(1+u_{s}-\kappa v)-u_{t}(v_{s}+\kappa u-\tau w).
\end{eqnarray*}

They derived the following identities for the necessary and sufficient
condition for the curve $r$ as a geodesic on the surface $x$.

\begin{equation*}
u(s,t_{0})=v(s,t_{0})=w(s,t_{0})=0\text{.}
\end{equation*}%
\begin{equation}
\phi _{1}(s,t_{0})=\phi _{3}(s,t_{0})=0  \label{geo1}
\end{equation}%
\begin{equation}
\phi _{2}(s,t_{0})\neq 0.  \label{geo2}
\end{equation}

\bigskip In \cite{asymptotic curve}, Bayram et al. derived the necessary and
sufficient condition for the curve $r$ as an asymptotic line on the surface $%
x$ as%
\begin{equation}
\frac{\partial \phi _{1}}{\partial s}(s,t_{0})-\kappa (s)\phi _{2}(s,t_{0})=0
\label{asymp}
\end{equation}

We express minimal surfaces with this parametric representation.

\section{Minimal surface family}

In this section, we derive necessary and sufficient condition for minimal
surface family passing the curve $r:r(s)\subset 
\mathbb{R}
^{3}$.

If we calculate the partial derivatives of \ $x:x(s,t)\subset 
\mathbb{R}
^{3}$ then we have, 
\begin{eqnarray*}
x_{t} &=&u_{t}T+v_{t}N+w_{t}B \\
x_{tt} &=&u_{tt}T+v_{tt}N+w_{tt}B \\
x_{s} &=&(1+u_{s}-\kappa v)T+(v_{s}+\kappa u-\tau w)N+(w_{s}+\tau v)B \\
x_{ss} &=&((1+u_{s}-\kappa v)_{s}-\kappa (v_{s}+\kappa u-\tau w))T \\
&&+((v_{s}+\kappa u-\tau w)_{s}+\kappa (1+u_{s}-\kappa v)-\tau (w_{s}+\tau
v))N \\
&&+((w_{s}+\tau v)_{s}+\tau (v_{s}+\kappa u-\tau w))B.
\end{eqnarray*}

\bigskip From \ref{isothermal} and \ref{harmonic} we have following
corollaries.

\begin{corollary}
If the surface $x$ is isothermal then 
\begin{align}
(1+u_{s}-\kappa v)^{2}+(v_{s}+\kappa u-\tau w)^{2}+(w_{s}+\tau v)^{2}&
=u_{t}^{2}+v_{t}^{2}+w_{t}^{2}  \label{isot} \\
(1+u_{s}-\kappa v)u_{t}+(v_{s}+\kappa u-\tau w)v_{t}+(w_{s}+\tau v)w_{t}& =0.
\label{isot2}
\end{align}
\end{corollary}

\begin{corollary}
If the surface $x$ is harmonic then%
\begin{align}
(1+u_{s}-\kappa v)_{s}-\kappa (v_{s}+\kappa u-\tau w)+u_{tt}& =0
\label{har1} \\
(v_{s}+\kappa u-\tau w)_{s}+\kappa (1+u_{s}-\kappa v)-\tau (w_{s}+\tau
v)+v_{tt}& =0  \label{har2} \\
(w_{s}+\tau v)_{s}+\tau (v_{s}+\kappa u-\tau w)+w_{tt}& =0.  \label{har3}
\end{align}
\end{corollary}

If we consider the surface satisfy the isothermal and harmonic conditions
together then we have the following theorem.

\begin{theorem}
Let $r:r(s)\subset 
\mathbb{R}
^{3}$ is be an arclenght parameterized curve. The surface \ref{surpen} is
minimal surface passing the curve $r$ if and only if there exist\ the smooth
functions $u(s,t),v(s,t),w(s,t)$ satisfying,%
\begin{equation*}
u(s,t_{0})=v(s,t_{0})=w(s,t_{0})=0
\end{equation*}%
\begin{eqnarray*}
(1+u_{s}-\kappa v)^{2}+(v_{s}+\kappa u-\tau w)^{2}+(w_{s}+\tau v)^{2}
&=&u_{t}^{2}+v_{t}^{2}+w_{t}^{2} \\
(1+u_{s}-\kappa v)u_{t}+(v_{s}+\kappa u-\tau w)v_{t}+(w_{s}+\tau v)w_{t} &=&0
\\
(1+u_{s}-\kappa v)_{s}-\kappa (v_{s}+\kappa u-\tau w)+u_{tt} &=&0 \\
(v_{s}+\kappa u-\tau w)_{s}+\kappa (1+u_{s}-\kappa v)-\tau (w_{s}+\tau
v)+v_{tt} &=&0 \\
(w_{s}+\tau v)_{s}+\tau (v_{s}+\kappa u-\tau w)+w_{tt} &=&0.
\end{eqnarray*}
\end{theorem}

\begin{proof}
If the conditon \ref{isopara} is satisfied then 
\begin{equation*}
x(s,t_{0})=r(s).
\end{equation*}%
Let the surface $x(s,t)$ is be minimal surface. From Lemma 1 and Corollary 1
there exist isothermal parameters and the surface $x$ is harmonic. Thus,
there exist\ the smooth functions $u(s,t),v(s,t),w(s,t)$ satisfying the
conditions \ref{isot}-\ref{har3}.

On the other hand, if there exist\ the smooth functions $%
u(s,t),v(s,t),w(s,t) $ satisfying the conditions \ref{isot}-\ref{har3} then
the surface $x(s,t)$ is isothermal and harmonic. Thus the surface $x(s,t)$
is minimal.
\end{proof}

\section{Minimal surface family passing a circle}

In this section, we investigate a minimal surface family passing a circle.

Let the curve $r:r(s)\subset 
\mathbb{R}
^{3}$ is be a circle,%
\begin{equation*}
r(s)=(4\cos \frac{s}{4},4\sin \frac{s}{4},0).
\end{equation*}%
The curvature and torsion of the circle $r$ are become 
\begin{equation*}
\kappa =\frac{1}{4},\text{ \ \ }\tau =0.
\end{equation*}

If we choose $u(s,t)=u(t)$, $v(s,t)=v(t)$, $w(s,t)=w(t)$ and $t_{0}=0$ then
from \ref{isot}-\ref{har3} we have 
\begin{eqnarray}
(1-\frac{v}{4})^{2}+\frac{1}{16}u^{2} &=&u_{t}^{2}+v_{t}^{2}+w_{t}^{2}
\label{c1} \\
(1-\frac{v}{4})u_{t}+\frac{1}{4}uv_{t} &=&0  \label{c2} \\
u_{tt}-\frac{1}{16}u &=&0  \label{c3} \\
v_{tt}-\frac{1}{16}v+\frac{1}{4} &=&0  \label{c4} \\
w_{tt} &=&0.  \label{c5}
\end{eqnarray}

From \ref{c5} 
\begin{equation*}
w(t)=ct.
\end{equation*}%
From \ref{c3} and \ref{c4}%
\begin{equation*}
u(t)=c_{2}(e^{\frac{t}{4}}-e^{-\frac{t}{4}})
\end{equation*}%
\begin{equation*}
v(t)=c_{3}e^{\frac{t}{4}}+c_{4}e^{-\frac{t}{4}}+4.
\end{equation*}%
From \ref{c1} and \ref{c2} 
\begin{equation*}
u(t)=0
\end{equation*}%
\begin{equation*}
v(t)=2(-1\pm \sqrt{1-c^{2}})e^{\frac{t}{4}}+2(-1\mp \sqrt{1-c^{2}})e^{-\frac{%
t}{4}}+4
\end{equation*}%
\begin{equation*}
w(t)=ct
\end{equation*}%
where $c\in \lbrack -1,1]$.

The Frenet frame of the circle $r$ is 
\begin{eqnarray*}
T(s) &=&(-\sin \frac{s}{4},\cos \frac{s}{4},0) \\
N(s) &=&(-\cos \frac{s}{4},-\sin \frac{s}{4},0) \\
B(s) &=&(0,0,1).
\end{eqnarray*}

Consequently, we derive the minimal surface family passing the circle $r$ as 
\begin{equation*}
x(s,t;\left\vert c\right\vert )=(2((1\mp \sqrt{1-c^{2}})e^{\frac{t}{4}%
}+(1\pm \sqrt{1-c^{2}})e^{-\frac{t}{4}})(\cos \frac{s}{4},\sin \frac{s}{4}%
),ct).
\end{equation*}

\begin{remark}
Its obvious that each members of this minimal surface family is a surface of
revolution. We know that only surface of \ revolution whic is minimal is a
piece of catenoid or a plane. Thus each member of this family is a piece of
catenoid or a plane.
\end{remark}

\begin{corollary}
For $c=\pm 1$, the circle $r$ is a geodesic on the surface 
\begin{equation*}
x(s,t;1)=(4\cosh \frac{t}{4}\cos \frac{s}{4},4\cosh \frac{t}{4}\sin \frac{s}{%
4},\pm t)
\end{equation*}
\end{corollary}

\begin{proof}
For $c=\pm 1$, \ref{geo1} and \ref{geo2} are satisfied. Thus, the circle $r$
is a geodesic on the surface $x(s,t;1)$.
\end{proof}

We show that some members of minimal surface family in Figure~\ref{fig:Figure1}-~\ref{fig:Figure4} for $c=1,$
$\frac{\sqrt{3}}{2},$ $\frac{\sqrt{5}}{3}$ and $0\leq s\leq 8\pi $, $-5\leq
t\leq 5$ .\\

\begin{figure}
\centering
\includegraphics{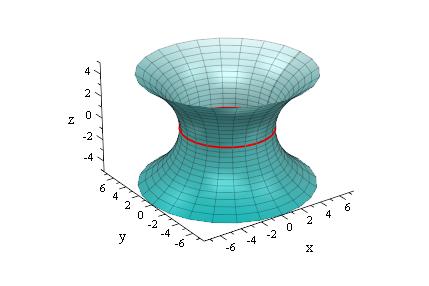}
\caption{The circle $r$ and minimal surface $x(s,t;0).$}
\label{fig:Figure1}
\end{figure}

\begin{figure}
\centering
\includegraphics{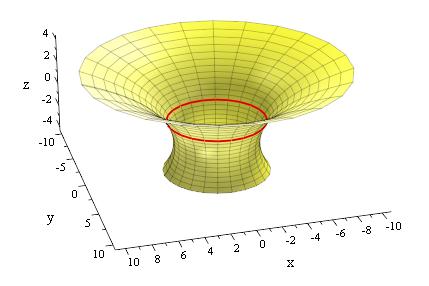}
\caption{The circle $r$ and minimal surface $x(s,t;\frac{\protect\sqrt{3}}{2}).$}
\label{fig:Figure2}
\end{figure}

\newpage

\begin{figure}
\centering
\includegraphics{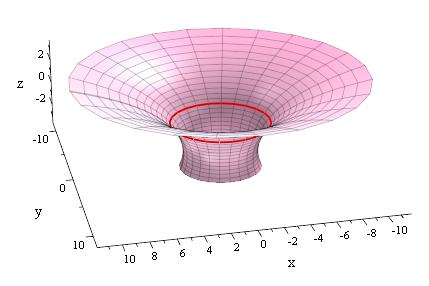}
\caption{The circle $r$ and minimal surface $x(s,t;\frac{\protect\sqrt{5}}{3}).$}
\label{fig:Figure3}
\end{figure}

\begin{figure}
\centering
\includegraphics{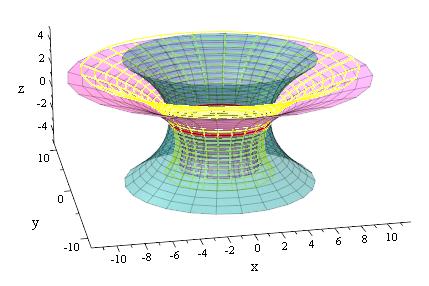}
\caption{Three members of minimal surface family $x(s,t;c).$}
\label{fig:Figure4}
\end{figure}

\newpage

\section{Minimal surface family passing a helix}

In this section, we investigate a minimal surface family passing a helix.

Let the curve $\alpha :\alpha (s)\subset 
\mathbb{R}
^{3}$ is be a helix,%
\begin{equation*}
\alpha (s)=(\frac{\sqrt{2}}{2}\cos s,\frac{\sqrt{2}}{2}\sin s,\frac{\sqrt{2}%
}{2}s).
\end{equation*}%
The curvature and torsion of the circle $r$ are become 
\begin{equation*}
\kappa (s)=\tau (s)=\frac{\sqrt{2}}{2}
\end{equation*}

If we choose $u(s,t)=u(t)$, $v(s,t)=v(t)$, $w(s,t)=w(t)$ and $t_{0}=0$ then
from \ref{isot}-\ref{har3} we have 
\begin{eqnarray}
(1-\frac{\sqrt{2}}{2}v)^{2}+\frac{1}{2}(u-w)^{2}+\frac{1}{2}v^{2}
&=&u_{t}^{2}+v_{t}^{2}+w_{t}^{2}  \label{h1} \\
(1-\frac{\sqrt{2}}{2}v)u_{t}+\frac{1}{2}(u-w)v_{t}+\frac{\sqrt{2}}{2}vw_{t}
&=&0  \label{h2} \\
u_{tt}-\frac{1}{2}(u-w) &=&0  \label{h3} \\
v_{tt}-v+\frac{\sqrt{2}}{2} &=&0  \label{h4} \\
w_{tt}+\frac{1}{2}(u-w) &=&0.  \label{h5}
\end{eqnarray}

From \ref{h3} and \ref{h5} 
\begin{equation*}
u_{tt}+w_{tt}=0.
\end{equation*}
\begin{equation*}
u+w=c_{1}t
\end{equation*}%
Thus,%
\begin{equation*}
u(t)=c_{2}(e^{t}-e^{-t})+\frac{1}{2}c_{1}t
\end{equation*}%
\begin{equation*}
w(t)=-c_{2}(e^{t}-e^{-t})+\frac{1}{2}c_{1}t
\end{equation*}%
From \ref{h4} 
\begin{equation*}
v(t)=c_{3}e^{t}+c_{4}e^{-t}+\frac{\sqrt{2}}{2}
\end{equation*}%
where $c_{2}+c_{3}=-\frac{\sqrt{2}}{2}$ .

From \ref{h1} and \ref{h2}%
\begin{equation*}
u(t)=\frac{1}{2}\cos c(-t+\sinh t)
\end{equation*}%
\begin{equation*}
v(t)=\sin c\sinh t-\frac{\sqrt{2}}{2}(\cosh t-1)
\end{equation*}%
\begin{equation*}
w(t)=-\frac{1}{4}\cos c(t+\sinh t).
\end{equation*}

The Frenet frame of the helix $\alpha $ is 
\begin{eqnarray*}
T(s) &=&(\frac{-\sqrt{2}}{2}\sin s,\frac{\sqrt{2}}{2}\cos s,\frac{\sqrt{2}}{2%
}), \\
N(s) &=&(-\cos s,-\sin s,0), \\
B(s) &=&(\frac{\sqrt{2}}{2}\sin s,-\frac{\sqrt{2}}{2}\cos s,\frac{\sqrt{2}}{2%
}).
\end{eqnarray*}

Consequently, we derive minimal surface family passing the helix $\alpha $
as 
\begin{eqnarray*}
y(s,t;c) &=&(\frac{\sqrt{2}}{2}\cosh t\cos s-\sinh t(\frac{\sqrt{2}}{2}\cos
c\sin s-\sin c\cos s), \\
&&\frac{\sqrt{2}}{2}\cosh t\sin s+\sinh t(\frac{\sqrt{2}}{2}\cos c\cos
s-\sin c\sin s), \\
&&\frac{\sqrt{2}}{2}t\cos c+\frac{\sqrt{2}}{2}s)
\end{eqnarray*}%
where $c\in 
\mathbb{R}
$ .

\begin{corollary}
For $c=0$, the helix $\alpha $ is a geodesic on the surface 
\begin{equation*}
y(s,t;0)=(\frac{\sqrt{2}}{2}\cosh t\cos s-\frac{\sqrt{2}}{2}\sinh t\sin s,%
\frac{\sqrt{2}}{2}\cosh t\sin s+\frac{\sqrt{2}}{2}\sinh t\cos s,\frac{\sqrt{2%
}}{2}t+\frac{\sqrt{2}}{2}s).
\end{equation*}
\end{corollary}

\begin{corollary}
For $c=\frac{\pi }{2}$, the helix $\alpha $ is an asymptotic line on the
surface 
\begin{equation*}
y(s,t;\frac{\pi }{2})=(\frac{\sqrt{2}}{2}\cosh t\cos s+\frac{\sqrt{2}}{2}%
\sinh t\cos s,\frac{\sqrt{2}}{2}\cosh t\sin s-\frac{\sqrt{2}}{2}\sinh t\sin
s,\frac{\sqrt{2}}{2}s).
\end{equation*}
\end{corollary}

\begin{proof}
For $c=\frac{\pi }{2}$, \ref{asymp} is satisfied. Thus, the helix $\alpha $
is a asymptotic line on the surface $y(s,t;\frac{\pi }{2})$.
\end{proof}

We show that some members of minimal surface family in Figure~\ref{fig:Figure5} -~\ref{fig:Figure8} for $c=0,$
$\frac{\pi }{4},$ $\frac{\pi }{2}$ and $0\leq s\leq 2\pi $, $-2\leq t\leq 2$.

\begin{figure}[here]
\centering
\includegraphics{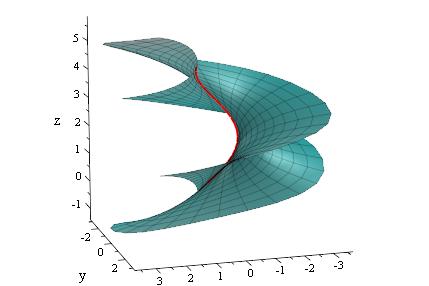}
\caption{The helix $\protect\alpha $ and minimal surface $y(s,t;0).$}
\label{fig:Figure5}
\end{figure}

\begin{figure}[here]
\centering
\includegraphics{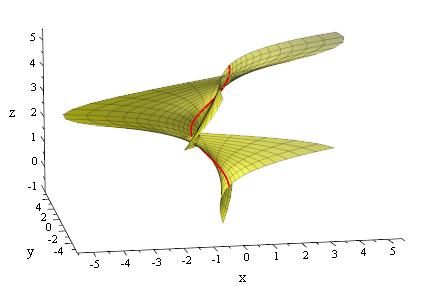}
\caption{The helix $\protect\alpha $ and minimal surface $y(s,t;\frac{\protect\pi }{4}).$}
\label{fig:Figure6}
\end{figure}
\newpage

\begin{figure}[here]
\centering
\includegraphics{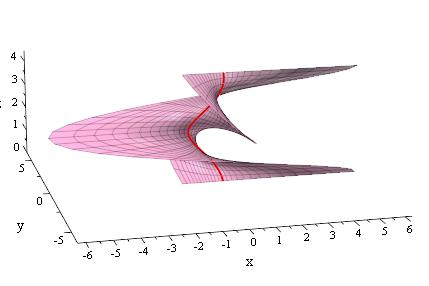}
\caption{The helix $\protect\alpha $ and minimal surface $y(s,t;\frac{\protect\pi }{2}).$}
\label{fig:Figure7}
\end{figure}

\begin{figure}[here]
\centering
\includegraphics{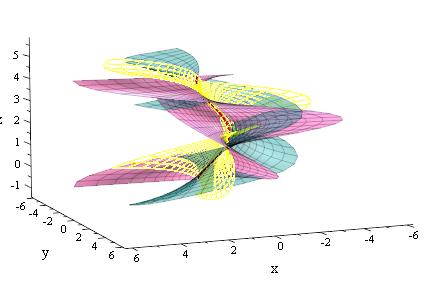}
\caption{Three members of minimal surface family $y(s,t;c)$.}
\label{fig:Figure8}
\end{figure}
\newpage

\section{Conclusion}

We derive necessary and sufficient condition for minimal surface passing a
given curve. Also, we present parametric representation of minimal surface
families passing a circle and a helix.

\end{document}